\documentclass{birkjour}
\usepackage{graphicx}
\usepackage{amsmath}
\usepackage{amsfonts}
\usepackage{amssymb}

\newtheorem{thm}{Theorem}[section]
\newtheorem{cor}[thm]{Corollary}
\newtheorem{lem}[thm]{Lemma}
\newtheorem{prop}[thm]{Proposition}
\theoremstyle{definition}
\newtheorem{defn}[thm]{Definition}
\newtheorem{rem}[thm]{Remark}

\theoremstyle{remark}

\numberwithin{equation}{section}
\newcommand{\norm}[1]{\left\Vert#1\right\Vert}
\newcommand{\abs}[1]{\left\vert#1\right\vert}

\newcommand{\eps}{\varepsilon}
\newcommand{\To}{\longrightarrow}

\newcommand{\hs}{\mathcal{H}}
\newcommand{\sh}{\mathcal{S}}

\newcommand{\Shq}{\mathcal{S}_q(M^2_{\Lambda},L^2(\mu))}

\def\<{\langle}
\def\>{\rangle}
\begin{document}

\title[]{Embeddings of M\"{u}ntz Spaces: Composition Operators}
\author{S.Waleed Noor}
\address{%
Abdus Salam School of Mathematical Sciences\\
New Muslim Town, Lahore, Pakistan}
\email{waleed\_math@hotmail.com}%

\thanks{}%
\begin{abstract}
Given a strictly increasing sequence $\Lambda=(\lambda_n)$ of nonegative real numbers,
with $\sum_{n=1}^\infty \frac{1}{\lambda_n}<\infty$, the M\"untz spaces $M_\Lambda^p$
are defined as the closure in $L^p([0,1])$ of the monomials $x^{\lambda_n}$. We discuss
how properties of the embedding $M_\Lambda^2\subset L^2(\mu)$, where $\mu$ is a finite
positive Borel measure on the interval $[0,1]$, have immediate consequences for composition
operators on $M^2_\Lambda$. We give criteria for composition operators to be bounded, compact,
or to belong to the Schatten--von Neumann ideals.
\end{abstract}
\maketitle

\section*{Introduction}

The M\"untz--Szasz Theorem states that,  if
$0=\lambda_0<\lambda_1<\dots<\lambda_n<\dots$ is an increasing sequence of
nonnegative real numbers, then the linear span of $x^{\lambda_n}$ is dense in
$C([0,1])$ if and only if $\sum_{n=1}^\infty \frac{1}{\lambda_n}=\infty$. When
$\sum_{n=1}^\infty \frac{1}{\lambda_n}<\infty$, the
closed linear span of the monomials $(x^{\lambda_n})_{n=0}^\infty$ in  $L^p([0,1])$ for $1\leq p<+\infty$
is a proper subspace of $L^p([0,1])$. These spaces, called M\"untz spaces
and denoted $M_\Lambda^p$, exhibit interesting properties that have not been
very much investigated. We refer principally to the monographies~\cite{Bor95,
Gur05}; recent results appear in~\cite{Alam08, Alam09, sp08, Dan10}.

In the paper \cite{Noor2011}, of which this work is a sequel, we investigated various properties and necessary conditions that allowed us to embed  the Hilbert M\"{u}ntz space $M^2_{\Lambda}$
into the Lebesgue space $L^2(\mu)$ for some positive measure $\mu$ on $[0,1]$. The boundedness, compactness and
Schatten ideal properties of this embedding were studied.

The purpose of this paper is to provide applications of the theory introduced in \cite{Noor2011} to composition operators on $M^2_\Lambda$. The plan of the paper is the following. After a section of preliminaries, we show in Section 2 that $M^p_\Lambda$ is not an invariant subspace of composition operators in general. It is then natural to study composition operators as mapping $M^p_\Lambda$ into $L^p([0,1])$. This is done in the sequel: sufficient conditions for composition operators to be bounded, compact or belong to Schatten ideals are obtained in Section 3, and necessary conditions in Section 4.

\section{Preliminaries}\label{se:prelim}
We denote by $m$ the Lebesgue measure on $[0,1]$. $L^p(\mu)$
shall be used to denote the space of Lebesgue integrable
functions of order $p\in[1,\infty]$ with respect to the measure $\mu$
on $[0,1]$. We will frequently use $L^p$ to mean $L^p(m)$, and
denote by $\norm{\cdot}_p$ and $||\cdot||_{L^p(\mu)}$ the norms
in $L^p(m)$ and $L^p(\mu)$ respectively.

Let us denote, for a set $S$ of nonnegative real numbers, the subspace
\[
L^p_S= \mathrm{closed \ span}\{x^t:t\in S\}\subset L^p.
\]
When clear from the context, we shall denote by $L_S$ the space $L^p_S$.
\begin{defn}Let $\Lambda$ be an increasing sequence of nonnegative real numbers with $\sum_{\lambda\in\Lambda}\frac{1}{\lambda}<\infty$. The M\"{u}ntz space $M^p_\Lambda$ is defined to be the space $L^p_\Lambda$.
\end{defn}
In this paper, $\Lambda$ shall always denote an increasing sequence of nonnegative real numbers with $\sum_{\lambda\in\Lambda}\frac{1}{\lambda}<\infty$. The functions in $M^p_\Lambda$ are continuous on $[0,1)$ and real analytic in (0,1). A feature of the M\"{u}ntz monomials $(x^{\lambda})_{\lambda\in\Lambda}$ is that they form a \emph{minimal system} in $M^p_\Lambda$, which means that for any $\lambda'\in\Lambda$
\[
\mathrm{dist}\,(x^{\lambda'},L_{\Lambda\backslash\{\lambda'\}})
=\inf_{g\in L_{\Lambda\backslash\{\lambda'\}}}||x^{\lambda'}-g||_{L^p}>0.
\]
This can easily be extended to show that if $\Lambda'\subset\Lambda$ is a finite subset, then
\begin{equation}\label{eq:minimal2}L_{\Lambda'}\cap L_{\Lambda\backslash\Lambda'}=\{0\}.
\end{equation}
The monograph \cite{Gur05} may be consulted for a discussion on the minimality of M\"{u}ntz monomials.

We shall need the Clarkson-Erd\"{o}s Theorem from \cite{Gur05}:
\begin{thm}\label{th:Cl-Erd}Assume that $\sum_k\frac{1}{\lambda_k}<\infty$ and $\inf_k(\lambda_{k+1}-\lambda_k)>0$. If $f\in M^p_\Lambda$ then there exist $b_k\in\mathbb{R}$ such that
\[
f(x)=\sum_{k=1}^\infty b_kx^{\lambda_k}\ \ \mathrm{for}\ x\in[0,1),
\]
where the series converges uniformly on compact subsets of $[0,1)$. Also, for any $\eps>0$, there is a constant $M>0$ such that
\begin{equation}\label{eq:Clark-Erdos}
|b_k|.||x^{\lambda_k}||_{L^p}\leq(1+\eps)^{\lambda_k}||f||_{L^p} \ \ \mathrm{if} \ k\geq M.
\end{equation}
\end{thm}

A sequence $\Lambda$ is called \emph{lacunary} if for some $\gamma>1$ we have $\lambda_{n+1}/\lambda_n\geq \gamma$ for
$n\geq 1$. More generally, $\Lambda$ is called \emph{quasilacunary} if for some increasing sequence $\{n_k\}$ of
integers with $N:=\sup_k(n_{k+1}-n_k)<\infty$ and some $\gamma>1$ we have $\lambda_{n_{k+1}}/\lambda_{n_k}\geq \gamma$.
The main feature of lacunarity is that the monomials $\lambda_n^{1/p}x^{\lambda_n}$ form a basis in each of the spaces $M^p_\Lambda$. In particular, the sequence $( \lambda_n^{1/2} x^{\lambda_n})_{n\ge1}$ forms a Riesz basis in~$M^2_\Lambda$.

If $T:\mathcal{E}\to\mathcal{F}$ is a bounded operator between
Banach spaces, we define by
$\norm{T}_e=\inf_\mathcal{K}\|T+\mathcal{K}\|$ the
$\emph{essential norm}$ of an operator, where the infimum is
taken over all compact operators
$\mathcal{K}:\mathcal{E}\to\mathcal{F}$. This norm measures
how far an operator is from being compact. In
particular, $T$ is compact if and only if $\|T\|_e=0$.

The \emph{Schatten--Von Neumann class} $\sh_q(\hs_1,\hs_2)$ is formed by the compact Hilbert space operators
$T:\hs_1\to\hs_2$ such that  $|T|=\sqrt{T^*T}:\hs_1\to\hs_1$ has a
family of eigenvalues $\{s_n(T)\}_{n=1}^\infty\in\ell_q$. If we define
\[
\|T\|_q= \left(
\sum_{n=1}^\infty s_n(T)^q
\right)^{1/q},
\]
then we obtain a quasinorm for $0<q<1$ and a norm for $q\ge 1$, with respect to which $\sh_q(\hs_1,\hs_2)$ is complete. It is immediate  that $\norm{T}_q\geq\norm{T}_{q'}$ for $q\leq q'$, hence~$\sh_q\subset \sh_{q'}$.

We now define $\Lambda$-embedding measures which were previously studied in \cite{Dan10} and~\cite{Noor2011}:
\begin{defn}
A positive measure $\mu$ on $[0,1]$ is
called $\Lambda_p$ -\emph{embedding}, if there is a constant
$C>0$ such that \begin{equation}\norm{g}_{L^p(\mu)}\leq
C\norm{g}_p\end{equation} for all polynomials $g\in
M^p_\Lambda$. Whenever $p$ is clear from the context,
we will remove subscript $p$ and use the notation
$\Lambda$-embedding.
\end{defn}

It follows easily from the definition (see~\cite{Dan10}) that a $\Lambda_p$-embedding measure $\mu$ has to satisfy $\mu({1})=0$. Therefore, as in Remark 2.5 of \cite{Dan10}, we may extend the embedding to all $f\in M^p_\Lambda$: if $\mu$ is $\Lambda_p$-embedding, then
$M^p_\Lambda\subset L^p(\mu)$ and $\norm{f}_{L^p(\mu)}\leq C\norm{f}_p$ for all $f\in M^p_\Lambda$.
For a $\Lambda_p$-embedding $\mu$ we denote by $i^p_\mu$ the embedding operator
$i^p_\mu:M^p_\Lambda\hookrightarrow L^p(\mu)$, which is bounded. If $0<\eps<1$, then the
interval $[1-\eps,1]$ will be denoted by $J_\eps$.

The next result is proved in \cite{Dan10} for $p=1$, but the extension to all $p\geq1$ is straightforward.

\begin{prop}\label{co:AbsCont h}
Let $M^p_\Lambda$ be a M\"{u}ntz space, and suppose there
exists $\delta>0$ such that $d\mu|_{J_\delta}=h\,dm|_{J_\delta}$
for some bounded measurable function $h$ with
$\lim_{t\rightarrow1}h(t)=a$. Then $i^p_\mu$ is bounded and
$||i_\mu^p||_e=a^{1/p}$.
\end{prop}

A new class of measures called sublinear measures was introduced in
\cite{Dan10}. There they were used to characterize embedding operators
$i_\mu:M^1_\Lambda\hookrightarrow L^1(\mu)$ for the class of
quasilacunary sequences $\Lambda$.

\begin{defn}
A measure $\mu$ is called \emph{sublinear} if there is a
constant $C>0$ such that for any $0<\eps<1$ we have
$\mu(J_\eps)\leq C\eps$. The smallest such $C$ will be denoted
by $\norm{\mu}_S$. The measure $\mu$ is called \emph{vanishing
sublinear} if $\lim_{\eps\rightarrow
0}\frac{\mu(J_\eps)}{\eps}=0$. Furthermore, a measure $\mu$ is called $\alpha$-\emph{sublinear} if $\mu(J_\eps)\leq C\eps^{\alpha}$ for some $\alpha>1$.
\end{defn}

The main embedding results in \cite{Noor2011} are contained in the next two theorems:

\begin{thm}\label{co:SublinearBounded}Let $\Lambda$ be lacunary and $\mu$ a positive measure on $[0,1]$. Then\\
$(i)$ $i_\mu^2$ is bounded if $\mu$ is sublinear.\\
$(ii)$ $i_\mu^2$ is compact if $\mu$ is vanishing sublinear.
\end{thm}

The above results are shown in \cite{Noor2011} to be true, after an interpolation argument, for all embeddings $i_\mu^p$ for $1\leq p\leq 2$.

  In \cite{Noor2011}, we also investigated conditions for measures that enabled the embedding $i^2_\mu$ to belong to $\sh_q$. We shall need the main results therein:

\begin{thm}\label{th:c_p embedding compact support}Let $\mu$ be a positive measure on $[0,1]$. Then $i^2_\mu\in \Shq$ for all $q>0$ if either of the following is true\\
$(i)$ $\mu$ has compact support in $[0,1)$,\\
$(ii)$ $\Lambda$ is quasilacunary and $\mu$ is $\,\alpha$-sublinear.
\end{thm}

Our goal is to apply these embedding results to composition operators. Recall that the pullback of a measure $\nu$
by $\phi$ is the measure $\phi^\ast\nu$ on $[0,1]$ defined by
$$\phi^\ast\nu(E)=\nu(\phi^{-1}(E))$$
for any Borel set $E$. If $g$ is a positive measurable function, then the formula
\[
\int^1_0g(\phi(x))dx=\int_{[0,1]}g\,d(\phi^* m)
\]
is easily checked on characteristic functions, hence the usual
argument extends it to all positive Borel functions on $[0,1]$. In particular,
if we define $\mu=\phi^\ast m$ and choose $g=|f|^p$ for some $f\in L^p(\mu)$, then the
map $J:L^p(\mu)\To L^p$ defined by $J(f)=f\circ\phi$ is
an isometry.

Let $\phi$ be a Borel function on $[0,1]$ such that
$\phi([0,1])\subset[0,1]$. The \emph{composition operator} $C_\phi$ is
defined as
\[
C_\phi(g)=g\circ\phi
\]
for all polynomials $g\in M^p_\Lambda$. Just as we did for $i_\mu^p$, we can extend $C_\phi=J\circ i^p_\mu$ to all $f\in M^p_\Lambda$. Since $J$ is an isometry, we obtain the following results for composition operators.

\begin{lem}\label{le:4parts}Define the measure $\mu=\phi^\ast m$. Then \\
(i) $C_\phi$ is bounded from $M^p_\Lambda$ to $L^p$ if and only if $\mu$
is a $\Lambda_p$-embedding measure.\\
(ii) $C_\phi$ is compact from $M^p_\Lambda$ to $L^p$ if and only if $i_\mu^p$
is  compact.\\
(iii) $ C_\phi\in\sh_q(M^2_\Lambda,L^2)$ if and only if $i_\mu^2\in\Shq$.
\end{lem}
\newpage{}
\section{M\"{u}ntz Spaces are not Invariant to Most Composition Operators}
It has already appeared above that we study composition operators defined on $M^p_\Lambda$, but whose range space is $L^p$. The reason is that M\"{u}ntz spaces are usually not invariant with respect to composition. This has already been noticed by Al Alam \cite{Alam09}, in the case of M\"{u}ntz space $M^\infty_\Lambda$, i.e. the closure of the span of monomials $x^{\lambda_n}$ in $L^\infty$, and operators $C_\phi$ with continuous $\phi$. The following result was proved therein.
\begin{prop}\label{prop:Alam}Let $\Lambda=(\lambda_k)_k\subset\mathbb{N}$ and $\sum_k\frac{1}{\lambda_k}<\infty$. Then\\
$(i)$ $C_\phi M^\infty_\Lambda\nsubseteq M^\infty_\Lambda$ if $\phi=\alpha x^m+\beta x^n$ with $\alpha,\beta\neq 0$ and $m,n\in\mathbb{N}$.\\
$(ii)$ $C_\phi M^\infty_\Lambda\nsubseteq M^\infty_\Lambda$ if $\phi$ is a polynomial with positive coefficients and more than one term.
\end{prop}

In this section we will significantly extend these results to other values of $p\geq1$ and functions $\phi$. We prove in Theorem \ref{th:non-invariance} that $C_\phi M^p_\Lambda\nsubseteq M^p_\Lambda$ whenever $\phi$ is a function of the form $c_1x^{s_1}+\ldots+c_lx^{s_l}$ with $c_i\in\mathbb{R}$ and $s_i\in\mathbb{R}^+$. These functions will be called \emph{real-exponent polynomials}. This generalizes Proposition \ref{prop:Alam} and $\Lambda$ may not even satisfy the \emph{gap condition} $\inf_k(\lambda_{k+1}-\lambda_k)>0$. If we assume the gap condition, then Theorem \ref{th:gap-invariance} generalizes Proposition \ref{prop:Alam}$(i)$ for arbitrary $\Lambda\subset\mathbb{R}^+$.

We start with a result of A. Schinzel \cite{Sch87}:

\begin{lem}\label{le:erdos}If $\phi$ is a polynomial with at least two terms and $\lambda\in\mathbb{N}$, then $\phi^\lambda$ has at least $\lambda+1$ terms.
\end{lem}

The next result is an analog of Lemma \ref{le:erdos} for real-exponent polynomials.
\begin{lem}\label{le:real erdos}If $\phi$ is a real-exponent polynomial with at least two terms and $\lambda\in\mathbb{N}$, then $\phi^\lambda$ has at least $\lambda+1$ terms.
\end{lem}
\begin{proof}Let $\phi(x)=c_1x^{s_1}+\ldots+c_lx^{s_l}$ with $c_i\in\mathbb{R}\backslash\{0\}$ and $s_i\in\mathbb{R}^+$. Considering $\mathbb{R}$ as a vector space over the rationals $\mathbb{Q}$, choose a basis $r_1,\ldots,r_\tau>0$ for the space spanned by $s_1,\ldots,s_l$ where $\tau\leq l$. Therefore
\[
s_i=\sum_{j=1}^\tau a_{ij}r_j \ \ \mathrm{for}\ \ i=1,\ldots,l
\]
where $a_{ij}\in\mathbb{Q}$. We may assume that $a_{ij}\in\mathbb{Z}$ by adjusting the $r_j$ suitably. We note that for any positive real number $N$, $\phi^\lambda$ has the same number of terms as $(x^N\phi)^\lambda$. So by choosing $N=b_1r_1+\ldots+b_\tau r_\tau$ with integers $b_j>|a_{ij}|$ for $i=1,\ldots,l$ and $j=1,\ldots,\tau$, we may also assume that each $a_{ij}r_j>0$ hence $a_{ij}\in\mathbb{N}$. We then obtain
\[
\phi(x)=\sum_{i=1}^lc_ix^{s_i}=\sum_{i=1}^lc_i(x^{r_1})^{a_{i1}}\ldots(x^{r_\tau})^{a_{i\tau}}.
\]
We define a polynomial $\psi$ in $\tau$ variables by
\[
\psi(Y_1,\ldots,Y_\tau)=\sum_{i=1}^lc_iY_1^{a_{i1}}\ldots Y_\tau^{a_{i\tau}}.
\]
 Define $\Phi$ to be the collection of monomial terms in $\phi^\lambda$ after reduction and cancelation, and $\Psi$ similarly for $\psi^\lambda$. Hence our goal is to prove that $\mathrm{card}\Phi\geq~\lambda+~1$. Since both $\phi$ and $\psi$ each have $l$ distinct monomial terms, the total number of possible products while computing $\phi^\lambda$ or $\psi^\lambda$ is $l^\lambda$.

 We claim that whenever two such products $p(x)=k.(x^{r_1})^{m_1}\ldots(x^{r_\tau})^{m_\tau}$ and $q(x)=k'.(x^{r_1})^{m_1'}\ldots(x^{r_\tau})^{m_\tau'}$ reduce (respectively cancel) in $\phi^\lambda$, the corresponding products $p_\psi(Y_1,\ldots,Y_\tau)=k.Y_1^{m_1}\ldots Y_\tau^{m_\tau}$ and $q_\psi(Y_1,\ldots,Y_\tau)=k'.Y_1^{m_1'}\ldots Y_\tau^{m_\tau'}$ also reduce (resp.\,cancel) in $\psi^\lambda$, where $m_j,m'_j\in\mathbb{N}$. Indeed, it is obvious that $p$ and $q$ combine (resp.\,cancel) if and only if $m_1r_1+\ldots+m_\tau r_\tau=m_1'r_1+\ldots+m_\tau' r_\tau$. Since $r_1,\ldots,r_\tau$ are linearly independent over $\mathbb{Q}$, this is possible if and only if $m_j=m_j'$ for $j=1,\ldots,\tau$. And this is equivalent to the reducing (resp.\,cancelling) of $p_\psi$ and $q_\psi$. This proves that $\mathrm{card}\Phi=\mathrm{card}\Psi$.

 Note that $\psi$ has at least two terms because $\phi$ has at least two terms. This implies that for some $1\leq j'\leq\tau$, $\psi$ as a polynomial in $Y_{j'}$ has at least two terms. Applying Lemma \ref{le:erdos} to $\psi'(Y_{j'}):=\psi(1,\ldots,Y_{j'},\ldots,1)=\sum_{i=1}^lc_iY_{j'}^{a_{ij'}}$, we see that $(\psi')^\lambda$ has at least $\lambda+1$ terms. Therefore $\psi^\lambda$ has at least $\lambda+1$ terms and $\mathrm{card}\Psi\geq\lambda+1$. Therefore $\mathrm{card}\Phi\geq\lambda+1$.
\end{proof}
The next lemma is a consequence of formula \eqref{eq:minimal2}.
\begin{lem}\label{lem:polynomial lemma}Let $\Lambda=(\lambda_k)_k$ and $\sum_k\frac{1}{\lambda_k}<\infty$. If a real-exponent polynomial $c_1x^{s_1}+\ldots+c_lx^{s_l}$ belongs to $M^p_\Lambda$ , then $s_1,\ldots,s_l\in\Lambda$.
\end{lem}
\begin{proof} Given $c_1x^{s_1}+\ldots+c_lx^{s_l}\in M^p_\Lambda=L_\Lambda$, suppose on the contrary that some subset $\Lambda'=\{s_{k_1},\ldots,s_{k_m}\}\subset\{s_1,\ldots,s_l\}$ does not belong to $\Lambda$ and $\{s_1,\ldots,s_l\}\backslash\Lambda'\subset\Lambda$. Then
\[
p(x)=(c_1x^{s_1}+\ldots+c_lx^{s_l})-(c_{k_1}x^{s_{k_1}}+\ldots+c_{k_m}x^{s_{k_m}})\in L_{\Lambda}
.\]
This implies that $c_{k_1}x^{s_{k_1}}+\ldots+c_{k_m}x^{s_{k_m}}=c_1x^{s_1}+\ldots+c_lx^{s_l}-p(x)\in L_{\Lambda'}\cap L_\Lambda$. But $L_{\Lambda'}\cap L_{\Lambda}=\{0\}$ by \eqref{eq:minimal2}, a contradiction.
\end{proof}

\begin{thm}\label{th:non-invariance}Suppose $\Lambda=(\lambda_k)_k\subset\mathbb{N}$ with $\sum_k\frac{1}{\lambda_k}<\infty$. If $\phi$ is a real-exponent polynomial with more than one term, then $C_\phi M^p_\Lambda\nsubseteq M^p_\Lambda$.
\end{thm}
\begin{proof}Let $\phi(x)=c_1x^{s_1}+\ldots+c_lx^{s_l}$ with $c_i\in\mathbb{R}\backslash\{0\}$ and $s_i\in\mathbb{R}^+$. Then for any $\lambda\in\Lambda$, we get $C_\phi(x^\lambda)=\phi^\lambda$ which has at least $\lambda+1$ terms by Lemma \ref{le:real erdos}. We may assume that these $\lambda+1$ terms are nonzero multiples of
\[
x^{s_1\lambda},x^{t_1},\ldots,x^{t_{\lambda-1}},x^{s_l\lambda} \ \ \mathrm{where} \ \ s_1\lambda<t_1<\ldots<t_{\lambda-1}<s_l\lambda
.\]
Suppose that $C_\phi M^p_\Lambda \subset M^p_\Lambda$, then Theorem \ref{lem:polynomial lemma} gives us $s_1\lambda,t_1,\ldots,t_{\lambda-1},s_l\lambda\in\Lambda$. We construct a subsequence $(\lambda_{k_j})_j$ of $\Lambda$ as follows: Let $\lambda_{k_1}=\lambda_1$ and inductively choose $\lambda_{k_j}$ such that $s_1\lambda_{k_j}>s_l\lambda_{k_{j-1}}$ for $j\geq2$. Then the sequence
\[
\Lambda^*:=\bigcup_{j=1}^\infty\{s_1\lambda_{k_j},t_1,\ldots,t_{\lambda_{k_j}-1},s_l\lambda_{k_j}\}
\]
is increasing and has distinct elements; moreover, $\Lambda^*\subset\Lambda$. So
\[
\sum_{k=1}^\infty\frac{1}{\lambda_k}\geq\sum_{s\in\Lambda^*}\frac{1}{s}\geq
\sum_{j=1}^\infty\sum_{i=1}^{\lambda_{k_j}+1}\frac{1}{s_l\lambda_{k_j}}\geq\sum_{j=1}^\infty\frac{1}{s_l}=\infty
\]
and hence the contradiction implies $C_\phi M^p_\Lambda\nsubseteq M^p_\Lambda$.
\end{proof}

\begin{cor}\label{cor:invariance}Let $\Lambda\subset\mathbb{N}$ and $\phi$ be a real-exponent polynomial. Then the following are equivalent:\\
$(i)$ $C_\phi M^p_\Lambda\subset M^p_\Lambda$\\
$(ii)$ $\phi(x)=\alpha x^{\eta}$ and $\Lambda=\Lambda.\{1,\eta,\eta^2,\ldots\}$ for some $0\leq\alpha\leq 1$ and $\eta\in\mathbb{R}^+$\\
$(iii)$ $C_\phi:M^p_\Lambda\to M^p_\Lambda$ is a bounded operator.
\end{cor}
\begin{proof}$(i)\Rightarrow(ii)$. Theorem \ref{th:non-invariance} implies that $\phi(x)=\alpha x^{\eta}$ for $\eta\in\mathbb{R}^+$ and $0\leq\alpha\leq 1$ because $\phi([0,1])\subset[0,1]$. Then $C_\phi^m(x^\lambda)=C_\phi^{m-1}(\alpha^\lambda x^{\lambda\eta})=\ldots=Kx^{\lambda\eta^m}\in M^p_\Lambda$ for any $\lambda\in\Lambda$, $m\in\mathbb{N}$ and some constant $K$. Hence $\lambda\eta^m\in\Lambda$ for all $\lambda\in\Lambda$ and $m\in\mathbb{N}$ by Lemma \ref{lem:polynomial lemma}. Therefore $\Lambda=\cup_{\lambda\in\Lambda}\lambda.\{1,\eta,\eta^2,\ldots\}=\Lambda.\{1,\eta,\eta^2,\ldots\}$.

$(ii)\Rightarrow(iii)$. Suppose $\phi(x)=\alpha x^{\eta}$ with $0\leq\alpha\leq 1$ and $\eta\in\mathbb{R}^+$. If $\alpha<1$, then $\mu=\phi^*m$ is supported on $[0,\alpha]$ and $d\mu|_{J_{1-\alpha}}=0$. Hence $||i_\mu^p||_e=0$ by Proposition \ref{co:AbsCont h} and $C_\phi=J\circ i_\mu^p$ is compact. For $\alpha=1$, the measure $\mu=\phi^*m$ satisfies
\[
\int_{J_\delta}fd\mu=\int_{\phi^{-1}(J_\delta)}f\circ\phi\,dm=\int_{J_\delta} f\cdot(\phi^{-1})'dm
=\int_{1-\delta}^1f(x)\,\eta^{-1}x^{\frac{1}{\eta}-1}dx
\]
for any continuous $f$ and $0<\delta<1$. Therefore $d\mu|_{J_\delta}=h\,dm|_{J_\delta}$ where $h(x)=\eta^{-1}x^{\frac{1}{\eta}-1}$ is bounded on $J_\delta$, and hence $C_\phi$ is bounded by Proposition \ref{co:AbsCont h}. Moreover, for any $\lambda\in\Lambda$ we see that $C_\phi x^\lambda=\alpha^\lambda x^{\lambda\eta}\in M^p_\Lambda$. Hence by the density of linear span of monomials $x^\lambda$ in $M^p_\Lambda$ and continuity of $C_\phi$, we get $C_\phi M^p_\Lambda\subset M^p_\Lambda$. The last part $(iii)\Rightarrow(i)$ is trivial.
\end{proof}

It is easy to see that Theorem \ref{th:non-invariance} and Corollary \ref{cor:invariance} can be extended to the case when $\Lambda\nsubseteq\mathbb{N}$, but contains a subsequence of integers. To go beyond this case, we need some preparation about \emph{real-exponent power series}.

\begin{lem}\label{lem:real power series}Suppose $f(x)=\sum_ka_kx^{s_k}$ is a series such that $(s_k)_k\subset\mathbb{R}^+$ is the finite union of sequences that satisfy the gap condition. Then $f$ is uniformly convergent on some interval $[0,\rho]$ if $L:=\limsup_k|a_k|^{1/s_k}<\infty$. Furthermore, if $f\equiv0$ on $[0,\rho_0]$ for $\rho_0\leq\rho$ then $a_k=0$ for all $k$.
\end{lem}
\begin{proof}It is sufficient to prove the first part for the case when $(s_k)_k$ itself satisfies the gap condition; in the general case, we can write $f$ as a finite sum of uniformly convergent series.

Since $|a_kx^{s_k}|^{1/s_k}=|a_k|^{1/s_k}|x|$, we get $\limsup_k|a_kx^{s_k}|^{1/s_k}<1$ if and only if
$|x|<L^{-1}$ (taking $L^{-1}=\infty$ if $L=0$). So, for $L|x|<1$, we get $\limsup_k|a_kx^{s_k}|^{1/s_k}<r<1$ for some $r$ and hence there exists a positive integer $N$ such that $|a_kx^{s_k}|^{1/s_k}<r$ for $k\geq N$. Therefore
\[
\sum_{k\geq N}^\infty|a_kx^{s_k}|\leq\sum_{k\geq N}^\infty r^{s_k}<\infty
\]
where the convergence follows from the ratio test and the gap condition because
\[
\lim_{k\to\infty}\frac{r^{s_{k+1}}}{r^{s_k}}=\lim_{k\to\infty}r^{s_{k+1}-s_k}\leq r^{\inf_k(s_{k+1}-s_k)}<1.
\]
So $f(x)$ converges absolutely for $L|x|<1$, and in particular converges uniformly on $[0,\rho]$ for some $\rho>0$.

For the second part, suppose on the contrary that $a_1$ is the first non-zero coefficient. We see that
\[
f(x)=\sum_{k\geq1} a_kx^{s_k}=a_1x^{s_1}(1+\sum_{k>1}\frac{a_k}{a_1}\,x^{s_k-s_1})
\]
where $(s_k-s_1)_k$ is again a union of finitely many series satisfying the gap condition and
\begin{align*}
\limsup_k\left|\frac{a_k}{a_1}\right|^{\frac{1}{s_k-s_1}}&
\leq\limsup_k(|a_k|^{\frac{1}{s_k}})^{\frac{s_k}{s_k-s_1}}.\limsup_k\left(\frac{1}{|a_1|}\right)^{\frac{1}{s_k-s_1}}=L<\infty
\end{align*}
hence $g(x)=1+\sum_{k>1}\frac{a_k}{a_1}\,x^{s_k-s_1}$ converges uniformly on some interval $[0,\rho_1]$. So $f(x)=a_1x^{s_1}g(x)=0$ on $[0,r]$, where $r=\min\{\rho_0,\rho_1\}$. Therefore $g=0$ on $(0,r]$ and hence on $[0,r]$ by continuity. A contradiction, since $g(0)=1$.
\end{proof}
\begin{thm}\label{th:gap-invariance}Suppose $\Lambda\subset\mathbb{R}^+$ with $\sum_k\frac{1}{\lambda_k}<\infty$ satisfies the gap condition $\inf_k(\lambda_{k+1}-\lambda_k)>0$. If $\phi=\alpha x^{\zeta_1}+\beta x^{\zeta_2}$ with $\alpha,\beta\neq0$ and $\zeta_1<\zeta_2\in\mathbb{R}^+$, then $C_\phi M^p_\Lambda\nsubseteq M^p_\Lambda$.
\end{thm}
\begin{proof} If $\Lambda\subset\mathbb{N}$, then Theorem \ref{th:non-invariance} proves the result. So we assume $\Lambda\nsubseteq\mathbb{N}$, hence there exists $\lambda\in\Lambda$ that is not an integer. Suppose that $C_\phi M^p_\Lambda\subset M^p_\Lambda$; then
\[
C_\phi(x^{\lambda})=(\alpha x^{\zeta_1}+\beta x^{\zeta_2})^{\lambda}=\alpha^{\lambda}x^{\lambda \zeta_1}(1+\frac{\beta}{\alpha}\,x^{\zeta_2-\zeta_1})^{\lambda}\in M^p_\Lambda.
\]
Hence by the binomial series we can represent $C_\phi(x^{\lambda})$ as
\[
C_\phi(x^{\lambda})(t)=\alpha^{\lambda}t^{\lambda \zeta_1}\sum_{k=0}^\infty a_kt^{k(\zeta_2-\zeta_1)}
=\alpha^{\lambda}\sum_{k=0}^\infty a_k t^{\lambda \zeta_1+k(\zeta_2-\zeta_1)}
\]
where the series converges for $|t|<|\frac{\alpha}{\beta}|^{\frac{1}{\zeta_2-\zeta_1}}$, in particular on $[0,\eta]$ for some $\eta<1$. The sequence of exponents $(\lambda \zeta_1+k(\zeta_2-\zeta_1))_k$ clearly satisfies the gap condition, while the coefficients
\[
a_k=\left(\frac{\beta}{\alpha}\right)^k\frac{\lambda(\lambda-1)(\lambda-2)\ldots(\lambda-k+1)}{k!}
\]
satisfy
\[
L_1:=\limsup_{k\to\infty}|a_k|^{1/\lambda \zeta_1+k(\zeta_2-\zeta_1)}<\infty.
\]
Similarly, by Theorem \ref{th:Cl-Erd} there exists a sequence of scalars $b_k\in\mathbb{R}$ such that
\[
C_\phi(x^{\lambda})(t)=\sum_{k=1}^\infty b_kt^{\lambda_k}
\]
and the series converges uniformly on compact subsets of $[0,1)$. By \eqref{eq:Clark-Erdos}, the coefficients $(b_k)_k$ satisfy
\[
L_2:=\limsup_{k\to\infty}|b_k|^{1/\lambda_k}
\leq\limsup_{k\to\infty}\,[(1+\eps)(2\lambda_k+1)^{1/2\lambda_k}||f||_{L^2}^{1/\lambda_k}]<\infty.
\]
Since both series representations coincide on $[0,\eta]$, the series defined by
\[
f(t)=\sum_{k=1}^\infty b_kt^{\lambda_k}-\alpha^{\lambda}\sum_{k=0}^\infty a_k t^{\lambda \zeta_1+k(\zeta_2-\zeta_1)}
=\sum_k\gamma_kt^{s_k}
\]
vanishes on $[0,\eta]$. Since $(s_k)_k$ is the union of two series satisfying the gap condition and $\limsup_k|\gamma_k|^{1/s_k}\leq L_1+L_2<\infty$, by Lemma \ref{lem:real power series} we get $\gamma_k=0$ for all $k$. Since $\lambda$ is not an integer, all the $a_k$ are non-zero; this implies that $\lambda \zeta_1+k(\zeta_2-\zeta_1)\in\Lambda$ for all $k$.
This contradicts the fact that $\sum_k\frac{1}{\lambda_k}<\infty$ and hence $C_\phi M^p_\Lambda\nsubseteq M^p_\Lambda$.
\end{proof}
\section{Composition Operators on $M^2_\Lambda$: direct results}
The next result is essentially contained in the work of Chalendar, Fricain and Timotin \cite{Dan10}:

\begin{prop}\label{prop:Timotin}Suppose the Borel function $\phi:[0,1]\To[0,1]$ satisfies the following: \\
(a) $\phi^{-1}(1)=\{x_1,\ldots,x_k\}$ is finite.\\
(b) There exists $\epsilon>0$ such that, for each
$i=1,...,k$, $\phi$ is continuous on \\
$(x_i-\epsilon,x_i+\epsilon)$, $\phi\in
C^1((x_i-\epsilon,x_i))$ and
$\phi\in C^1((x_i,x_i+\epsilon))$.\\
(c) $\phi_-'(x_i)>0$ and $\phi'_+(x_i)<0$ for all $i=1,\dots,k$. \\
($\phi_-'(x)$ and $\phi_+'(x)$ denote the left and right derivatives at $x$ respectively, which may be infinite).\\
(d) There exists $\alpha<1$ such that, if
$x\notin\cup^k_{i=1}(x_i-\epsilon,x_i+\epsilon)$, then
$\phi(x)<\alpha$.

Then $ C_\phi:M^2_\Lambda\To L^2$ is bounded and $||C_\phi||_e=\sum_{i=1}^kL(x_i)$, where
\[L(x_i)=
\left\{
  \begin{array}{ll}
    \frac{1}{\phi_-'(x_i)}+\frac{1}{|\phi'_+(x_i)|}  \  \ if \ x_i\in(0,1), \\
    \frac{1}{\phi_-'(x_i)} \ \ \ \ \ \ \ \ \ \ \ \ \ \ if \ x_i=1, \\
    \frac{1}{|\phi'_+(x_i)|} \ \ \ \ \ \ \ \ \ \ \ \ \  if \ x_i=0.
  \end{array}
\right.
\]
In particular, if $\phi_-'(x_i)=\infty$ and $\phi'_+(x_i)=-\infty$ for all $i=1,\dots,k$, then $C_\phi$ is compact.
\end{prop}
 We intend to go beyond the regularity assumptions in Proposition \ref{prop:Timotin}.

\begin{defn}
If $\phi:[0,1]\to[0,1]$ is a Borel function and $\alpha=\mathrm{ess}\sup_{[0,1]}\phi$, then a point $x\in[0,1]$ is an \emph{essential point of maximum} for $\phi$ if ess $\sup_E\phi=\alpha$ for every neighborhood $E$ of $x$. Denote by $\mathfrak{M}_\phi$ the set of all essential points of maxima of $\phi$, and by $V_\eps$ the neighborhood of $\mathfrak{M}_\phi$ defined for each $\eps>0$ by
\[
V_\eps=\{x\in[0,1]:\mathrm{dist}(x,\mathfrak{M}_\phi)<\eps\}.
\]
\end{defn}

\begin{lem}\label{lem:V lemma}
The following statements are true:\\
$(i)$ $\mathfrak{M}_\phi$ is non-empty and closed,\\
$(ii)$ $\mathrm{ess}\sup\phi|_{[0,1]\backslash V_\eps}<\alpha$ for all $\eps>0$,\\
$(iii)$ for every $\eps>0$ there exists a $\delta_0>0$ such that $\phi^{-1}([\alpha-\delta,\alpha])\subset V_\eps$ almost everywhere whenever $0<\delta<\delta_0$.
\end{lem}
\begin{proof}$(i)$. If $\mathfrak{M}_\phi$ were empty, then every point $x\in[0,1]$ would have a neighborhood $\mathcal{N}_x$ such that ess $\sup_{\mathcal{N}_x}\phi<\alpha$ and all such $\mathcal{N}_x$ would cover $[0,1]$. Choosing a finite subcover so that $\cup_{k=1}^m\mathcal{N}_{x_k}=[0,1]$, we see that
\[\mathrm{ess}  \sup_{[0,1]}\phi=\max_k\{\mathrm{ess}  \sup_{\mathcal{N}_{x_k}}\phi\}<\alpha.\]
The contradiction yields $\mathfrak{M}_\phi\neq \varnothing$. To prove that $\mathfrak{M}_\phi$ is closed, consider
the set $\mathcal{S}:=\cup_{x\in[0,1]\backslash \mathfrak{M}_\phi}\mathcal{N}_x$, where $\mathcal{N}_x$ again represents a neighborhood of $x$ on which ess~$\sup_{\mathcal{N}_x}\phi<\alpha$.  So clearly $\mathcal{S}$ is  open, and $\mathcal{S}\cap\mathfrak{M}_\phi=\varnothing$ since otherwise some $\mathcal{N}_{x'}$ for $x'\in[0,1]\backslash \mathfrak{M}_\phi $ would contain an essential point of maximum. Hence $\mathcal{S}=[0,1]\backslash \mathfrak{M}_\phi$ and $\mathfrak{M}_\phi$ is closed.

For $(ii)$, suppose that $\mathrm{ess}\sup\phi|_{[0,1]\backslash V_{\eps'}}=\alpha$ for some $\eps'>0$. Then the argument in the proof of $(i)$ applied to the compact set $[0,1]\backslash V_{\eps'}$, shows that it contains an essential point of maximum.

Finally for $(iii)$, it follows from $(ii)$ that for every $\eps>0$ there exists a $\delta_0>0$ such that ess $\sup\phi|_{[0,1]\backslash V_\eps}<\alpha-\delta_0<\alpha$ and hence
\[\phi^{-1}([\alpha-\delta,\alpha])=\{x\in[0,1]:\alpha-\delta\leq \phi(x)\leq \alpha\}\subset V_\eps\] except possibly for a subset of measure $0$, whenever $0<\delta<\delta_0$.
\end{proof}

We recall that the left and right derivatives of $\phi$ at the point $y$ are defined as
\[\label{eq:D-}D_-^i(y)=\liminf_{t\to y-}\frac{\phi(y)-\phi(t)}{y-t}\]
\[\label{eq:D+}D_+^i(y)=\liminf_{t\to y+}\frac{\phi(y)-\phi(t)}{y-t}\]
\[D_-^s(y)=\limsup_{t\to y-}\frac{\phi(y)-\phi(t)}{y-t}\]
\[D_+^s(y)=\limsup_{t\to y+}\frac{\phi(y)-\phi(t)}{y-t}\]
respectively.

Suppose $\phi:[0,1]\to[0,1]$ is a Borel function such that $\alpha=\mathrm{ess}\sup_{[0,1]}\phi<1$. Then it is easy to show that the measure defined by $\mu=\phi^*m$ has support in $[0,\alpha]$. In fact
\[
\mu((\alpha,1])=\int_{(\alpha,1]}d(\phi^*m)
=\int_{\phi^{-1}((\alpha,1])}dm=m(\phi^{-1}(\alpha,1])=0.
\]
Hence in this case $i^2_\mu\in\sh_q$ by Theorem \ref{th:c_p embedding compact support}$(i)$. Therefore $C_\phi\in\sh_q$ by Lemma \ref{le:4parts}, so  from here onwards we assume that $\alpha=\mathrm{ess}\sup_{[0,1]}\phi=1$.

Since changing the values of $\phi$ on a set of measure zero does not effect $\mu=\phi^*m$, whenever $m(\mathfrak{M}_\phi)=0$, one may take $\phi\equiv 1$ on $\mathfrak{M}_\phi$. This will be assumed in the rest of the paper.
\begin{lem}\label{le:s-lemma}Suppose $\phi$ is a Borel function with $\mathfrak{M}_\phi=\{x_1,\ldots,x_k\}$ and $\mu=\phi^*m$. If for some $s\geq 1$ there exists an $\eps>0$ and a constant $c>0$ such that
\[
|x-x_i|\leq c|\phi(x)-1|^s \ \ \mathrm{whenever} \ \ |x-x_i|<\eps
\]
for all $i=1,\ldots,k$, then there exists a $\delta_0>0$ such that $\mu(J_\delta)\leq 2kc\delta^s$ whenever $0<\delta<\delta_0$.
\end{lem}
\begin{proof}By Lemma \ref{lem:V lemma}(iii), there exists a $\delta_0>0$ such that $\phi^{-1}(J_\delta)\subset V_{\eps}$ almost everywhere whenever $0<\delta<\delta_0$. Since $\sup_{\phi^{-1}(J_\delta)}|\phi(x)-1|\leq\delta$, we get
\begin{align*}
m(\phi^{-1}(J_\delta))&\leq\sum_{i=1}^km(\phi^{-1}(J_\delta)\cap\{\mathrm{dist}(x,x_i)<\eps\})
\leq2\sum_{i=1}^k\sup_{\phi^{-1}(J_\delta)\cap\{|x-x_i|<\eps\}}|x-x_i| \\
&\leq2\sum_{i=1}^k\sup_{\phi^{-1}(J_\delta)\cap\{|x-x_i|<\eps\}}c|\phi(x)-1|^s\leq 2kc\delta^s.
\end{align*}
Therefore we get
\[
\mu(J_\delta)=\int_{J_\delta}d\mu=\int_{J_\delta}d(\phi^\ast m)
=\int_{\phi^{-1}(J_\delta)}dm=m(\phi^{-1}(J_\delta))\leq 2kc\delta^s
\]
whenever $0<\delta<\delta_0$.
\end{proof}

We arrive at the main theorem that gives necessary conditions for composition operators on $M^2_\Lambda$ to be bounded, compact or in $\sh_q$.

\begin{thm}\label{thm:MainComposition} Let $\Lambda$ be lacunary and $\mathfrak{M}_\phi=\{x_1,\dots,x_k\}$.\\
$(i)$ If $D_-^i>0$ and $D_+^s<0$ on $\mathfrak{M}_\phi$, then $C_\phi:M^2_\Lambda\to L^2$ is bounded.\\
$(ii)$ If $D_-^i=+\infty$ and $D_+^s=-\infty$ on $\mathfrak{M}_\phi$, then $C_\phi:M^2_\Lambda\to L^2$ is compact.\\
$(iii)$ If for some $\eps>0$, $\beta>1$ and constant $c$ we have
\begin{equation}\label{eq:comp3}\abs{x-x_i}\leq c\abs{\phi(x)-1}^\beta  \ \ \forall \ \ |x-x_i|<\eps\end{equation}
for $i=1,\ldots,k$, then $C_\phi\in\mathcal{S}_q(M^2_{\Lambda},L^2)$ $\forall$ $q>0$.
\end{thm}
\begin{proof}(i) The hypothesis about the derivatives implies that for some constant $M>0$ there exists an $\eps>0$ such that
\[\frac{|\phi(x)-1|}{|x-x_i|}\geq M>0\ \ \Longleftrightarrow \ |x-x_i|\leq M^{-1}|\phi(x)-1|\] whenever $|x-x_i|<\eps$ for all $i=1,\ldots,k$. Hence by Lemma \ref{le:s-lemma}, we get $\mu(J_\delta)\leq 2kM^{-1}\delta$ for $0<\delta<\delta_0$ .
Therefore $\mu$ is sublinear and $i_\mu^2$ is bounded by Theorem \ref{co:SublinearBounded} (i). So Lemma \ref{le:4parts} implies that $ C_\phi:M^2_\Lambda\to L^2$ is bounded.

(ii) By our hypothesis, for any $M>0$ there exists an $\eps>0$ such that
\[\frac{|\phi(x)-1|}{|x-x_i|}\geq M \ \ \Longleftrightarrow \ |x-x_i|\leq M^{-1}|\phi(x)-1|\]
whenever $|x-x_i|<\eps$. And for every such $\eps>0$ there exists a $\delta_0>0$ such that $\mu(J_\delta)\leq 2kM^{-1}\delta$ whenever $0<\delta<\delta_0$ by Lemma \ref{le:s-lemma}. Therefore $\frac{\mu(J_\delta)}{\delta}\to 0$ as $\delta\to 0$. So the measure $\mu$ defined above is a vanishing sublinear measure hence $i_\mu^2:M^2_\Lambda\to L^2(\mu)$ is compact by Theorem \ref{co:SublinearBounded} (ii), and so is $ C_\phi=J\circ i_\mu^2$.

(iii) Applying Lemma \ref{le:s-lemma} directly to condition \eqref{eq:comp3}, we get $\mu(J_\delta)\leq 2kc\delta^\beta$ whenever $0<\delta<\delta_0$. Hence by Theorem \ref{th:c_p embedding compact support}$(ii)$, $i_\mu\in
\mathcal{S}_q(M^2_{\Lambda},L^2(\mu))$ for all $q>0$. So
$C_\phi\in\mathcal{S}_q(M^2_{\Lambda},L^2)$ for all $q>0$.
\end{proof}

\begin{rem}If $\psi\in L^\infty$ then these results still hold true for the weighted composition operator $M_\psi\circ \,C_\phi$ where $M_\psi$ is the multiplication operator with symbol $\psi$, which is a bounded operator on $L^2$.
\end{rem}

\section{Composition Operators on $M^2_\Lambda$: Inverse results}We conclude by presenting some results that serve as converses to the boundedness and compactness theorems given above for composition operators on $M^2_\Lambda$. We shall need the following two lemmas.

\begin{lem}\label{le:liminf lemma}Let $\mu$ be a positive measure on $[0,1]$. Then the following hold:\\
$(i)$ If $i_\mu^2$ is bounded, then $\liminf_{\delta\to 0}\frac{\mu(J_\delta)}{\delta}<\infty$  \\
$(ii)$ If $\,i_\mu^2$ is compact, then $\liminf_{\delta\to 0}\frac{\mu(J_\delta)}{\delta}= 0$ .
\end{lem}
\begin{proof}(i) Suppose $\mu$ is $\Lambda_2$-embedding. Since $\lim_{n\to\infty}(1-\frac{1}{\lambda_n})^{\lambda_n}=\frac{1}{e}$, there exists an integer $N$ such that, for all $n\geq N$ and for all $x\in[1-\frac{1}{\lambda_n},1]$, we have $x^{\lambda_n}\geq\frac{1}{3}$. It follows that for all $n\geq N$
\[
\frac{1}{3^2}\mu(J_{1/\lambda_n})\leq\int_{J_{1/\lambda_n}}x^{2\lambda_n}d\mu\leq ||i_\mu^2||^2\int_0^1x^{2\lambda_n}dx=\frac{||i_\mu^2||^2}{2\lambda_n+1}.
\]
Therefore for all $n\geq N$, we have
\[
\mu(J_{1/\lambda_n})\leq\frac{3^2||i^2_\mu||^2}{\lambda_n}\ \Longleftrightarrow \ \frac{\mu(J_{1/\lambda_n})}{1/\lambda_n}\leq 9||i_\mu^2||^2.
\]
This implies that $\liminf_{\delta\to 0}\frac{\mu(J_\delta)}{\delta}<\infty$.

(ii) Choosing $f_n(x)=\lambda_n^{1/2}x^{\lambda_n}$, we see that
\[
\langle f_n\,,x^{\lambda_k}\rangle=\int_{[0,1]}\lambda^{1/2}_n x^{\lambda_n+\lambda_k}dx=\frac{\lambda_n^{1/2}}{\lambda_n+\lambda_k+1}\To 0
\]
as $n\to\infty$ for all $k\in\mathbb{N}$. Noting that $||f_n||_{L^2}$ is bounded and the linear span of the sequence $(x^{\lambda_k})_k$ is dense in $M^2_\Lambda$, it follows that $f_n\to 0$ weakly in $M^2_\Lambda$, as $n\to\infty$. If $i^2_\mu$ is compact, this implies that $(i^2_\mu\, f_n)_n$ converges strongly to $0$ in $L^2(\mu)$ and hence $||f_n||_{L^2(\mu)}\to 0$ as $n\to 0$. Therefore
\[
||f_n||^2_{L^2(\mu)}=\int_{[0,1]}\lambda_n x^{2\lambda_n}d\mu\geq\int_{J_{1/\lambda_n}}\lambda_n x^{2\lambda_n}d\mu\geq(1-\frac{1}{\lambda_n})^{2\lambda_n}\frac{\mu(J_{1/\lambda_n})}{1/\lambda_n}.
\]
Since $(1-\frac{1}{\lambda_n})^{2\lambda_n}\to e^{-2}$ as $n\to\infty$, we get
\[
\frac{\mu(J_{1/\lambda_n})}{1/\lambda_n}\To 0 \ \ \ \mathrm{as} \ \ \ n\to\infty
\]
and the result follows.
\end{proof}
The next lemma might be compared to Lemma \ref{le:s-lemma}.
\begin{lem}\label{le:eta lemma}Suppose $\phi:[0,1]\to[0,1]$ is a Borel function and $\mu=\phi^*m$. If for some $x_0\in[0,1]$ with $\phi(x_0)=1$ and $\eta>0$, there exists an $\eps>0$ such that
\[
x_0-x>\frac{1}{\eta}\,(1-\phi(x)) \ \ \ \mathrm{whenever}  \ \ 0<x_0-x<\eps,
\]
then $\mu(J_\delta)\geq\frac{\delta}{\eta}$ for $0<\delta<\eta\,\eps$.
\end{lem}
\begin{proof}Since there exists an $\eps>0$ such that for $0<x_0-x<\eps$, we have
\[
\frac{1-\phi(x)}{x_0-x}<\eta \ \ \Longleftrightarrow \ \ 1-\phi(x)<\eta(x_0-x).
\]
Then suppose $0<\delta<\delta_0=\eta\,\eps$. If $0<x_0-x<\frac{\delta}{\delta_0}\,\eps=\frac{\delta}{\eta}$ then
$1-\phi(x)<\eta\frac{\delta}{\eta}=\delta$ which implies $\phi(x)>1-\delta$. So $\phi^{-1}(J_\delta)$ contains the interval $(x_0-\frac{\delta}{\eta},x_0)$ of Lebesgue measure $\frac{\delta}{\eta}$. Therefore
$$
m(\phi^{-1}(J_\delta))\geq\frac{\delta}{\eta}\  \Rightarrow \ \mu(J_\delta)\geq\frac{\delta}{\eta} \
\Rightarrow \ \frac{\mu(J_\delta)}{\delta}\geq\frac{1}{\eta}.\eqno\qedhere
$$
\end{proof}

For the partial converses to parts (i) and (ii) of Theorem \ref{thm:MainComposition}, we need neither lacunarity nor any assumption on $\mathfrak{M}_\phi$:

\begin{thm}\label{th:Converses}Suppose $\phi:[0,1]\to[0,1]$ is a Borel function, and $\phi(x_0)=1$ for some $x_0\in[0,1]$. \\
$(i)$ If  $C_\phi$ is bounded, then $D_-^s(x_0)>0$ and $D_+^i(x_0)<0$. \\
$(ii)$ If $C_\phi$ is compact, then $D_-^s(x_0)=+\infty$ and $D_+^i(x_0)=-\infty$.
\end{thm}

\begin{proof}(i) Suppose on the contrary that either $D_-^s(x_0)=0$ or $D_+^i(x_0)=0$. We shall deduce a contradiction for one of these cases since both are analogous. So suppose $D_-^s(x_0)=0$, that is
\[
\lim_{x\to x_0-}\frac{1-\phi(x)}{x_0-x}=0.
\]
For each $\eta>0$ there exists an $\eps>0$ such that for $0<x_0-x<\eps$, we have
\[
\frac{1-\phi(x)}{x_0-x}<\eta.
\]
Therefore by Lemma \ref{le:eta lemma} we get $\frac{\mu(J_\delta)}{\delta}\geq\frac{1}{\eta}$ whenever $\delta<\eta\,\eps$. So $\frac{\mu(J_\delta)}{\delta}\to +\infty$ as $\delta\to 0$. Since $C_\phi$ is bounded we get that $\mu$ is $\Lambda_2$-embedding by Lemma \ref{le:4parts}. This leads to a contradiction since Lemma \ref{le:liminf lemma} gives $\liminf_{\delta\to 0}\frac{\mu(J_\delta)}{\delta}<\infty$.

For (ii), suppose to the contrary that either $D_-^s(x_0)<+\infty$ or $D_+^i(x_0)>-\infty$ for some $x_0\in\mathfrak{M}_\phi$. Again due to similarities we shall deal with one case. So suppose $D_-^s(x_0)<\infty$, that is
\[
\lim_{x\to x_0-}\frac{1-\phi(x)}{x_0-x}<\infty.
\]
So there exists a $\zeta>0$ and an $\eps>0$ such that for $0<x_0-x<\eps$, we have
\[
\frac{1-\phi(x)}{x_0-x}<\zeta.
\]
Therefore by Lemma \ref{le:eta lemma} we get $\frac{\mu(J_\delta)}{\delta}\geq\frac{1}{\zeta}$ for $\delta<\zeta\,\eps$. This contradicts Lemma \ref{le:liminf lemma} because $i_\mu^2$ is compact by Lemma \ref{le:4parts}.
\end{proof}

\begin{cor}Suppose $\phi$ is a polynomial with $\phi^{-1}(1)$ non-empty. Then $C_\phi$ is not compact, and if it is bounded then $\phi^{-1}(1)\subset\{0,1\}$.
\end{cor}
\begin{proof}If $C_\phi$ is bounded and some $x_0\in\phi^{-1}(1)$ is an interior point of $[0,1]$, then clearly $x_0$ must be a local maximum and hence $\phi'(x_0)=0$. This contradicts Theorem \ref{th:Converses} (i) and hence $\phi^{-1}(1)\subset\{0,1\}$. Similarly, by part (ii) of the theorem we get the conclusion that $C_\phi$ can never be compact because $\phi$ is differentiable everywhere.
\end{proof}

\textbf{Acknowledgments:} The author wishes to thank his supervisor Professor Dan Timotin for the many interesting ideas he shared and discussions we had.

\bibliographystyle{amsplain}

\end{document}